\documentclass[a4paper,english]{article}
\usepackage[a4paper]{geometry}
\usepackage{amsthm,amsmath,amssymb,mathtools}
\usepackage{enumerate,enumitem,csquotes,verbatim}
\usepackage{setspace}
\usepackage[numbers]{natbib}
\usepackage{hyperref}
\usepackage{color}
\usepackage{tikz}
\usepackage{multirow}

\theoremstyle{plain}
\newtheorem*{theorem*}{Theorem}
\newtheorem{theorem}{Theorem}
\newtheorem{lemma}[theorem]{Lemma}

\newtheorem*{claim*}{Claim}

\newtheorem{problem}[theorem]{Problem}

\theoremstyle{definition}

\theoremstyle{remark}
\newtheorem*{remark}{Remark}

\def\N{\mathbb{N}}
\def\Z{\mathbb{Z}}

\newcommand*{\abs}[1]{\lvert#1\rvert}

\title{On the Existence of Zero-Sum Perfect Matchings of Complete Graphs}

\author{Teeradej Kittipassorn\thanks{\,Department of Mathematics and Computer Science, Faculty of Science, Chulalongkorn University, Bangkok 10330, Thailand; \texttt{teeradej.k@chula.ac.th}.}
  \and Panon Sinsap\thanks{\,Department of Mathematics and Computer Science, Faculty of Science, Chulalongkorn University, Bangkok 10330, Thailand; \texttt{panon.sinsap@gmail.com}.}}

\begin{document}
\maketitle

\begin{abstract}
    In this paper, we prove that given a $2$-edge-coloured complete graph $K_{4n}$ that has the same number of edges of each colour, we can always find a perfect matching with an equal number of edges of each colour. This solves a problem posed by Caro, Hansberg, Lauri, and Zarb. The problem is also independently solved by Ehard, Mohr, and Rautenbach.
\end{abstract}

\section{Introduction}

Note that in this paper, we will use the word `matching' when in fact we mean `perfect matching'.

For an edge-colouring function $f:E(G)\to S$ of a graph $G$ where $S\subseteq\Z$ and a subgraph $H$ of $G$, if $\sum_{e\in E(H)} f(e)=0$ then $H$ is called a \emph{zero-sum subgraph} of $G$.

The research in zero-sum problems can be traced back to the three theorems that give them the algebraic foundation. These are the Erdős-Ginzberg-Ziv Theorem~\cite{MR3618568}, the Cauchy-Davenport Theorem~\cite{Dave}, and Chevalley's Theorem~\cite{MR3069644}. Early zero-sum results concern with the sum taken in additive group $\Z_k$, the area is called Zero-sum Ramsey Theory. This theory studies the zero-sum Ramsey number $R(G,\Z_k)$ which is the smallest number $n$ such that in every $\Z_k$-edge-colouring of $r$-uniform hypergraph on $n$ vertices $K_n^{(r)}$ there exists a zero-sum modulo $k$ copy of $G$. It also studies the zero-sum bipartite Ramsey number $B(G,\Z_k)$ which is the smallest number $n$ such that in every $\Z_k$-edge-colouring of $K_{n,n}$ there exists a zero-sum modulo $k$ copy of $G$. For more complete developments of the topic consult~\cite{MR1261194,MR1388634}.

In~\cite{MR1295790}, Caro gave the complete characterization of the zero-sum modulo $2$ Ramsey number $R(G,\Z_2)$. In~\cite{MR1647802}, Caro and Yuster gave the characterization of zero-sum modulo $2$ bipartite Ramsey numbers. Along with~\cite{MR1694458} and \cite{MR3090519}, these four papers completely solved the zero-sum Ramsey theory over $\Z_2$. Caro and Yuster~\cite{MR3436948} were the first to consider zero-sum problems over $\Z$. Recently, several variants of the zero-sum problems have been studied (see, e.g.~\cite{MR3962015,MR3861785}).

In~\cite{caro2020zerosum}, Caro, Hansberg, Lauri, and Zarb had studied zero-sum subgraphs where $S=\{-1,1\}$ from various host graphs and various kinds of subgraphs. They then proved what they call the `Master Theorem' which covers many results of this kind. However there is a remarkable variation of zero-sum subgraph problem that has not been decided by their work in that paper. So they posed it at the end of the paper and the problem is the following:

\begin{problem}
Suppose $f:E(K_{4n})\to\{-1,1\}$ is such that it is a zero-sum graph. Does a zero-sum matching always exist?\label{p1}
\end{problem}

Observe that this problem essentially wants us to find a matching that has an equal number of edges that was assigned with $-1$ and $1$ out of a complete graph of degree $4n$ that had been assigned an equal number of $-1$ and $1$ to their edges. This allow us to discard the arithmetic meanings of $-1$ and $1$ and replace them with general colour names. In this paper, we choose to use black and red.

Our main result is the following theorem whose proof is equivalent to the solution of Problem~\ref{p1}.

\begin{theorem}
For any $2$-edge-colouring of $K_{4n}$ with an equal number of edges of each colour, there exists a matching with an equal number of edges of each colour.\label{t2}
\end{theorem}

Recently, Problem \ref{p1} had been independently resolved by Ehard, Mohr, and Rautenbach~\cite{ehard2020low}.

\section{Terminology}
To facilitate the language of our proof, we introduce the following notations and terminologies.

For a graph $G$, we define $\mathcal{M}(G)$ to denote the set of all matchings in $G$.

In $K_{4n}$, $n\in\N$, we define the operation $S:\mathcal{M}(G)\times V(G)^4\to \mathcal{M}(G)$ by
    $$S(M,u,v,x,y)=
    \begin{cases}
        (V(M),E(M)\cup\{ux,vy\}-\{uv,xy\}) &\text{if }uv\in M \text{ and }xy\in M,\\
        M &\text{otherwise.}
    \end{cases}$$
This operation will be called a \emph{swapping} (see Figure \ref{f1}).

\begin{figure}[h]
\centering
\begin{tikzpicture}
\draw[gray, thick] (0,1.5) -- (1.5,1.5);
\draw[gray, thick] (0,0) -- (1.5,0);
\filldraw[black] (0,1.5) circle (2pt) node[anchor=east] {$u$};
\filldraw[black] (1.5,1.5) circle (2pt) node[anchor=west] {$v$};
\filldraw[black] (0,0) circle (2pt) node[anchor=east] {$x$};
\filldraw[black] (1.5,0) circle (2pt) node[anchor=west] {$y$};
\filldraw[] (0.75,-0.25) node[anchor=north] {$M$};

\draw[->,very thick] (2.5,0.75) -- (3.5,0.75);

\draw[gray, thick] (4.5,1.5) -- (4.5,0);
\draw[gray, thick] (6,1.5) -- (6,0);
\filldraw[black] (4.5,1.5) circle (2pt) node[anchor=east] {$u$};
\filldraw[black] (6,1.5) circle (2pt) node[anchor=west] {$v$};
\filldraw[black] (4.5,0) circle (2pt) node[anchor=east] {$x$};
\filldraw[black] (6,0) circle (2pt) node[anchor=west] {$y$};
\filldraw[] (5.25,-0.25) node[anchor=north] {$S(M,u,v,x,y)$};

\end{tikzpicture}
\caption{The result of a swapping. All other edges in the matching stays the same.}
\label{f1}
\end{figure}
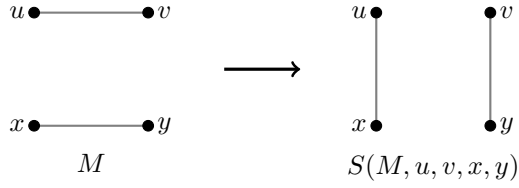

If $M$ is a matching extracted from $2$-edge-coloured $K_{4n}$, $V_B (M)$ will denote the set of all vertices of $M$ incident to a black edge in $M$, while $V_R (M)$ will denote the same thing for red.

For any disjoint subsets $S,T\subseteq V(G)$, $E(S,T)$ denotes the set of all edges with one endpoint in $S$ and one endpoint in $T$.

Let $b(M)$ and $r(M)$ denote the number of black edges and red edges in $M$ respectively.

Lastly, sometimes we will shorten the phrase `the difference between the numbers of edges of each colour' to merely `the difference'. As there will be only one kind of difference in this work, this should cause no ambiguity.

Now we have all the terminologies needed for our proof.

\section{Proof of the Theorem}

We first state an important observation as a lemma.

\begin{lemma}
For $M\in\mathcal{M}(G)$, if there are more red edges than black edges in $E(V_R (M),V_B (M))$, then we can make a swapping that will increase the number of red edges and decrease the number of black edges in $M$ by $1$ each. In particular, if there are more black edges than red edges in $M$ and there are more red edges than black edges in $E(V_R (M),V_B (M))$, then we can make a swapping that will reduce the difference by $2$.
\end{lemma}

\begin{proof}

Consider two edges of $M$, one black and one red, $uv$ and $xy$ respectively. The edges joining between the vertices of these two edges will be among the following six varieties

\begin{figure}[h]
    \centering
    \begin{tikzpicture}
    \draw[black,very thick] (0,1) -- (0,0);
    \draw[black,very thick] (1,1) -- (1,0);
    \draw[black,very thick] (0,1) -- (1,0);
    \draw[black,very thick] (1,1) -- (0,0);
    \draw[black,dashed,very thick] (0,1) -- (1,1);
    \draw[red,dashed,very thick] (0,0) -- (1,0);
    \filldraw[black] (0,1) circle (2pt) node[anchor=east] {$u$};
    \filldraw[black] (1,1) circle (2pt) node[anchor=west] {$v$};
    \filldraw[black] (0,0) circle (2pt) node[anchor=east] {$x$};
    \filldraw[black] (1,0) circle (2pt) node[anchor=west] {$y$};
    
    \draw[black,very thick] (3,1) -- (3,0);
    \draw[red,very thick] (4,1) -- (4,0);
    \draw[black,very thick] (3,1) -- (4,0);
    \draw[black,very thick] (4,1) -- (3,0);
    \draw[black,dashed,very thick] (3,1) -- (4,1);
    \draw[red,dashed,very thick] (3,0) -- (4,0);
    \filldraw[black] (3,1) circle (2pt) node[anchor=east] {$u$};
    \filldraw[black] (4,1) circle (2pt) node[anchor=west] {$v$};
    \filldraw[black] (3,0) circle (2pt) node[anchor=east] {$x$};
    \filldraw[black] (4,0) circle (2pt) node[anchor=west] {$y$};
    
    \draw[black,very thick] (6,1) -- (6,0);
    \draw[red,very thick] (7,1) -- (7,0);
    \draw[black,very thick] (6,1) -- (7,0);
    \draw[red,very thick] (7,1) -- (6,0);
    \draw[black,dashed,very thick] (6,1) -- (7,1);
    \draw[red,dashed,very thick] (6,0) -- (7,0);
    \filldraw[black] (6,1) circle (2pt) node[anchor=east] {$u$};
    \filldraw[black] (7,1) circle (2pt) node[anchor=west] {$v$};
    \filldraw[black] (6,0) circle (2pt) node[anchor=east] {$x$};
    \filldraw[black] (7,0) circle (2pt) node[anchor=west] {$y$};
    
    \end{tikzpicture}
    \caption{The first three varieties.}
\end{figure}

\begin{figure}[h]
    \centering
    \begin{tikzpicture}
    \draw[black,very thick] (0,1) -- (0,0);
    \draw[black,very thick] (1,1) -- (1,0);
    \draw[red,very thick] (0,1) -- (1,0);
    \draw[red,very thick] (1,1) -- (0,0);
    \draw[black,dashed,very thick] (0,1) -- (1,1);
    \draw[red,dashed,very thick] (0,0) -- (1,0);
    \filldraw[black] (0,1) circle (2pt) node[anchor=east] {$u$};
    \filldraw[black] (1,1) circle (2pt) node[anchor=west] {$v$};
    \filldraw[black] (0,0) circle (2pt) node[anchor=east] {$x$};
    \filldraw[black] (1,0) circle (2pt) node[anchor=west] {$y$};
    
    \draw[black,very thick] (3,1) -- (3,0);
    \draw[red,very thick] (4,1) -- (4,0);
    \draw[red,very thick] (3,1) -- (4,0);
    \draw[red,very thick] (4,1) -- (3,0);
    \draw[black,dashed,very thick] (3,1) -- (4,1);
    \draw[red,dashed,very thick] (3,0) -- (4,0);
    \filldraw[black] (3,1) circle (2pt) node[anchor=east] {$u$};
    \filldraw[black] (4,1) circle (2pt) node[anchor=west] {$v$};
    \filldraw[black] (3,0) circle (2pt) node[anchor=east] {$x$};
    \filldraw[black] (4,0) circle (2pt) node[anchor=west] {$y$};
    
    \draw[red,very thick] (6,1) -- (6,0);
    \draw[red,very thick] (7,1) -- (7,0);
    \draw[red,very thick] (6,1) -- (7,0);
    \draw[red,very thick] (7,1) -- (6,0);
    \draw[black,dashed,very thick] (6,1) -- (7,1);
    \draw[red,dashed,very thick] (6,0) -- (7,0);
    \filldraw[black] (6,1) circle (2pt) node[anchor=east] {$u$};
    \filldraw[black] (7,1) circle (2pt) node[anchor=west] {$v$};
    \filldraw[black] (6,0) circle (2pt) node[anchor=east] {$x$};
    \filldraw[black] (7,0) circle (2pt) node[anchor=west] {$y$};
    
    \end{tikzpicture}
    \caption{The first three varieties.}
\end{figure}

Observe that in the first three varieties the number of black edges are no fewer than the numbers of red edges and that the resulting matchings from $S(M,u,v,x,y)$ and $S(M,u,v,y,x)$ will never reduce the difference between the numbers of black edges and red edges that we had from $M$.

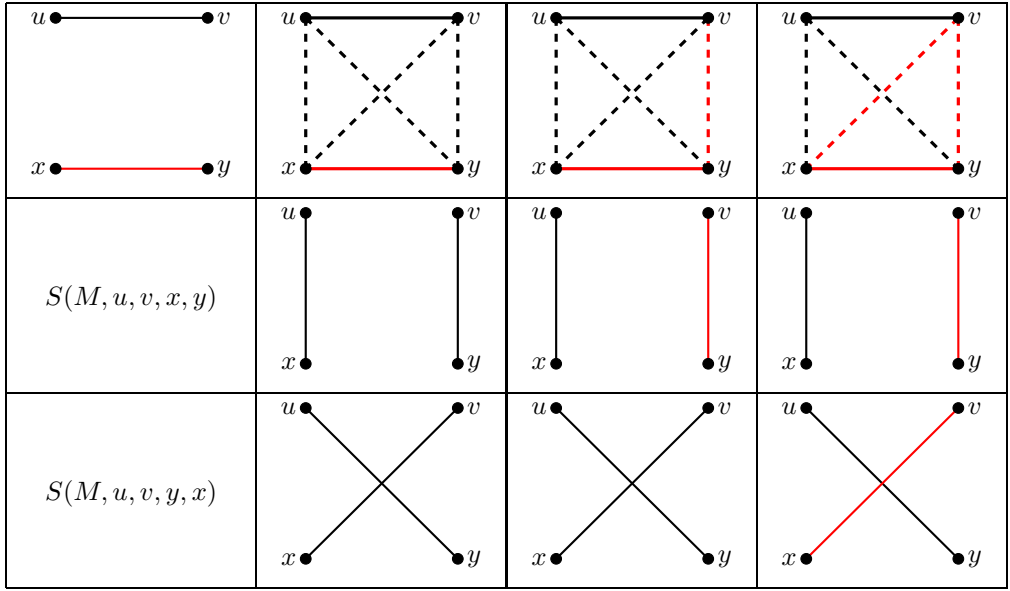
\begin{center}
\begin{tabular}{|c | c | c | c|} 
\hline
\begin{tikzpicture}
    \draw[black, thick] (0,2) -- (2,2);
    \draw[red, thick] (0,0) -- (2,0);
    \filldraw[black] (0,2) circle (2pt) node[anchor=east] {$u$};
    \filldraw[black] (2,2) circle (2pt) node[anchor=west] {$v$};
    \filldraw[black] (0,0) circle (2pt) node[anchor=east] {$x$};
    \filldraw[black] (2,0) circle (2pt) node[anchor=west] {$y$};
\end{tikzpicture}
 & 
 \begin{tikzpicture}
    \draw[black,dashed,very thick] (0,2) -- (0,0);
    \draw[black,dashed,very thick] (2,2) -- (2,0);
    \draw[black,dashed,very thick] (0,2) -- (2,0);
    \draw[black,dashed,very thick] (2,2) -- (0,0);
    \draw[black,very thick] (0,2) -- (2,2);
    \draw[red,very thick] (0,0) -- (2,0);
    \filldraw[black] (0,2) circle (2pt) node[anchor=east] {$u$};
    \filldraw[black] (2,2) circle (2pt) node[anchor=west] {$v$};
    \filldraw[black] (0,0) circle (2pt) node[anchor=east] {$x$};
    \filldraw[black] (2,0) circle (2pt) node[anchor=west] {$y$};
 \end{tikzpicture} 
 &
 \begin{tikzpicture}
    \draw[black,dashed,very thick] (0,2) -- (0,0);
    \draw[red,dashed,very thick] (2,2) -- (2,0);
    \draw[black,dashed,very thick] (0,2) -- (2,0);
    \draw[black,dashed,very thick] (2,2) -- (0,0);
    \draw[black,very thick] (0,2) -- (2,2);
    \draw[red,very thick] (0,0) -- (2,0);
    \filldraw[black] (0,2) circle (2pt) node[anchor=east] {$u$};
    \filldraw[black] (2,2) circle (2pt) node[anchor=west] {$v$};
    \filldraw[black] (0,0) circle (2pt) node[anchor=east] {$x$};
    \filldraw[black] (2,0) circle (2pt) node[anchor=west] {$y$};
 \end{tikzpicture} 
 &
 \begin{tikzpicture}
    \draw[black,dashed,very thick] (0,2) -- (0,0);
    \draw[red,dashed,very thick] (2,2) -- (2,0);
    \draw[black,dashed,very thick] (0,2) -- (2,0);
    \draw[red,dashed,very thick] (2,2) -- (0,0);
    \draw[black,very thick] (0,2) -- (2,2);
    \draw[red,very thick] (0,0) -- (2,0);
    \filldraw[black] (0,2) circle (2pt) node[anchor=east] {$u$};
    \filldraw[black] (2,2) circle (2pt) node[anchor=west] {$v$};
    \filldraw[black] (0,0) circle (2pt) node[anchor=east] {$x$};
    \filldraw[black] (2,0) circle (2pt) node[anchor=west] {$y$};
 \end{tikzpicture}
 \\ 
 \hline
 \multirow{-5}*{$S(M,u,v,x,y)$}
 & 
 \begin{tikzpicture}
    \draw[black, thick] (0,2) -- (0,0);
    \draw[black, thick] (2,2) -- (2,0);
    \filldraw[black] (0,2) circle (2pt) node[anchor=east] {$u$};
    \filldraw[black] (2,2) circle (2pt) node[anchor=west] {$v$};
    \filldraw[black] (0,0) circle (2pt) node[anchor=east] {$x$};
    \filldraw[black] (2,0) circle (2pt) node[anchor=west] {$y$};
\end{tikzpicture} 
 & 
\begin{tikzpicture}
    \draw[black, thick] (0,2) -- (0,0);
    \draw[red, thick] (2,2) -- (2,0);
    \filldraw[black] (0,2) circle (2pt) node[anchor=east] {$u$};
    \filldraw[black] (2,2) circle (2pt) node[anchor=west] {$v$};
    \filldraw[black] (0,0) circle (2pt) node[anchor=east] {$x$};
    \filldraw[black] (2,0) circle (2pt) node[anchor=west] {$y$};
\end{tikzpicture} 
 & 
\begin{tikzpicture}
    \draw[black, thick] (0,2) -- (0,0);
    \draw[red, thick] (2,2) -- (2,0);
    \filldraw[black] (0,2) circle (2pt) node[anchor=east] {$u$};
    \filldraw[black] (2,2) circle (2pt) node[anchor=west] {$v$};
    \filldraw[black] (0,0) circle (2pt) node[anchor=east] {$x$};
    \filldraw[black] (2,0) circle (2pt) node[anchor=west] {$y$};
\end{tikzpicture} 
 \\ 
 \hline
 \multirow{-5}*{$S(M,u,v,y,x)$}
 & 
\begin{tikzpicture}
    \draw[black, thick] (0,2) -- (2,0);
    \draw[black, thick] (2,2) -- (0,0);
    \filldraw[black] (0,2) circle (2pt) node[anchor=east] {$u$};
    \filldraw[black] (2,2) circle (2pt) node[anchor=west] {$v$};
    \filldraw[black] (0,0) circle (2pt) node[anchor=east] {$x$};
    \filldraw[black] (2,0) circle (2pt) node[anchor=west] {$y$};
\end{tikzpicture} 
 & 
\begin{tikzpicture}
    \draw[black, thick] (0,2) -- (2,0);
    \draw[black, thick] (2,2) -- (0,0);
    \filldraw[black] (0,2) circle (2pt) node[anchor=east] {$u$};
    \filldraw[black] (2,2) circle (2pt) node[anchor=west] {$v$};
    \filldraw[black] (0,0) circle (2pt) node[anchor=east] {$x$};
    \filldraw[black] (2,0) circle (2pt) node[anchor=west] {$y$};
\end{tikzpicture} 
 & 
\begin{tikzpicture}
    \draw[black, thick] (0,2) -- (2,0);
    \draw[red, thick] (2,2) -- (0,0);
    \filldraw[black] (0,2) circle (2pt) node[anchor=east] {$u$};
    \filldraw[black] (2,2) circle (2pt) node[anchor=west] {$v$};
    \filldraw[black] (0,0) circle (2pt) node[anchor=east] {$x$};
    \filldraw[black] (2,0) circle (2pt) node[anchor=west] {$y$};
\end{tikzpicture} 
 \\ 
 \hline
\end{tabular}
\end{center}

While in the last three varieties the numbers of red edges are no fewer than the numbers of black edges and that at least one of the resulting matchings from $S(M,u,v,x,y)$ or $S(M,u,v,y,x)$ reduces the difference between the numbers of black edges and red edges that we had from $M$.

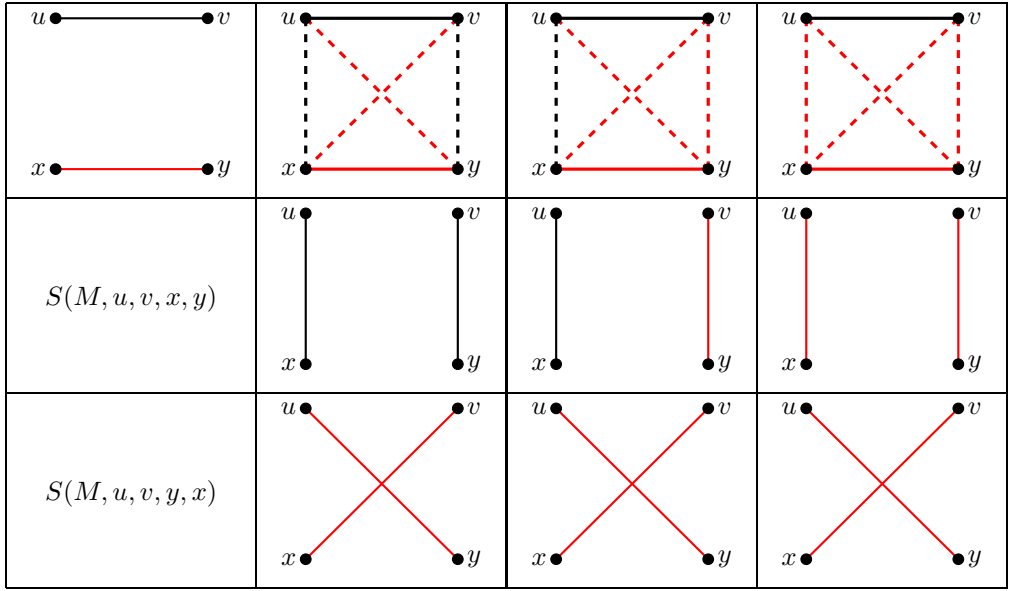
\begin{center}
\begin{tabular}{|c | c | c | c|} 
\hline
\begin{tikzpicture}
    \draw[black, thick] (0,2) -- (2,2);
    \draw[red, thick] (0,0) -- (2,0);
    \filldraw[black] (0,2) circle (2pt) node[anchor=east] {$u$};
    \filldraw[black] (2,2) circle (2pt) node[anchor=west] {$v$};
    \filldraw[black] (0,0) circle (2pt) node[anchor=east] {$x$};
    \filldraw[black] (2,0) circle (2pt) node[anchor=west] {$y$};
\end{tikzpicture}
 & 
 \begin{tikzpicture}
    \draw[black,dashed,very thick] (0,2) -- (0,0);
    \draw[black,dashed,very thick] (2,2) -- (2,0);
    \draw[red,dashed,very thick] (0,2) -- (2,0);
    \draw[red,dashed,very thick] (2,2) -- (0,0);
    \draw[black,very thick] (0,2) -- (2,2);
    \draw[red,very thick] (0,0) -- (2,0);
    \filldraw[black] (0,2) circle (2pt) node[anchor=east] {$u$};
    \filldraw[black] (2,2) circle (2pt) node[anchor=west] {$v$};
    \filldraw[black] (0,0) circle (2pt) node[anchor=east] {$x$};
    \filldraw[black] (2,0) circle (2pt) node[anchor=west] {$y$};
 \end{tikzpicture} 
 &
 \begin{tikzpicture}
    \draw[black,dashed,very thick] (0,2) -- (0,0);
    \draw[red,dashed,very thick] (2,2) -- (2,0);
    \draw[red,dashed,very thick] (0,2) -- (2,0);
    \draw[red,dashed,very thick] (2,2) -- (0,0);
    \draw[black,very thick] (0,2) -- (2,2);
    \draw[red,very thick] (0,0) -- (2,0);
    \filldraw[black] (0,2) circle (2pt) node[anchor=east] {$u$};
    \filldraw[black] (2,2) circle (2pt) node[anchor=west] {$v$};
    \filldraw[black] (0,0) circle (2pt) node[anchor=east] {$x$};
    \filldraw[black] (2,0) circle (2pt) node[anchor=west] {$y$};
 \end{tikzpicture} 
 &
 \begin{tikzpicture}
    \draw[red,dashed,very thick] (0,2) -- (0,0);
    \draw[red,dashed,very thick] (2,2) -- (2,0);
    \draw[red,dashed,very thick] (0,2) -- (2,0);
    \draw[red,dashed,very thick] (2,2) -- (0,0);
    \draw[black,very thick] (0,2) -- (2,2);
    \draw[red,very thick] (0,0) -- (2,0);
    \filldraw[black] (0,2) circle (2pt) node[anchor=east] {$u$};
    \filldraw[black] (2,2) circle (2pt) node[anchor=west] {$v$};
    \filldraw[black] (0,0) circle (2pt) node[anchor=east] {$x$};
    \filldraw[black] (2,0) circle (2pt) node[anchor=west] {$y$};
 \end{tikzpicture}
 \\ 
 \hline
 \multirow{-5}*{$S(M,u,v,x,y)$}
 & 
 \begin{tikzpicture}
    \draw[black, thick] (0,2) -- (0,0);
    \draw[black, thick] (2,2) -- (2,0);
    \filldraw[black] (0,2) circle (2pt) node[anchor=east] {$u$};
    \filldraw[black] (2,2) circle (2pt) node[anchor=west] {$v$};
    \filldraw[black] (0,0) circle (2pt) node[anchor=east] {$x$};
    \filldraw[black] (2,0) circle (2pt) node[anchor=west] {$y$};
\end{tikzpicture} 
 & 
\begin{tikzpicture}
    \draw[black, thick] (0,2) -- (0,0);
    \draw[red, thick] (2,2) -- (2,0);
    \filldraw[black] (0,2) circle (2pt) node[anchor=east] {$u$};
    \filldraw[black] (2,2) circle (2pt) node[anchor=west] {$v$};
    \filldraw[black] (0,0) circle (2pt) node[anchor=east] {$x$};
    \filldraw[black] (2,0) circle (2pt) node[anchor=west] {$y$};
\end{tikzpicture} 
 & 
\begin{tikzpicture}
    \draw[red, thick] (0,2) -- (0,0);
    \draw[red, thick] (2,2) -- (2,0);
    \filldraw[black] (0,2) circle (2pt) node[anchor=east] {$u$};
    \filldraw[black] (2,2) circle (2pt) node[anchor=west] {$v$};
    \filldraw[black] (0,0) circle (2pt) node[anchor=east] {$x$};
    \filldraw[black] (2,0) circle (2pt) node[anchor=west] {$y$};
\end{tikzpicture} 
 \\ 
 \hline
 
 \multirow{-5}*{$S(M,u,v,y,x)$} 
 & 
\begin{tikzpicture}
    \draw[red, thick] (0,2) -- (2,0);
    \draw[red, thick] (2,2) -- (0,0);
    \filldraw[black] (0,2) circle (2pt) node[anchor=east] {$u$};
    \filldraw[black] (2,2) circle (2pt) node[anchor=west] {$v$};
    \filldraw[black] (0,0) circle (2pt) node[anchor=east] {$x$};
    \filldraw[black] (2,0) circle (2pt) node[anchor=west] {$y$};
\end{tikzpicture} 
 & 
\begin{tikzpicture}
    \draw[red, thick] (0,2) -- (2,0);
    \draw[red, thick] (2,2) -- (0,0);
    \filldraw[black] (0,2) circle (2pt) node[anchor=east] {$u$};
    \filldraw[black] (2,2) circle (2pt) node[anchor=west] {$v$};
    \filldraw[black] (0,0) circle (2pt) node[anchor=east] {$x$};
    \filldraw[black] (2,0) circle (2pt) node[anchor=west] {$y$};
\end{tikzpicture} 
 & 
\begin{tikzpicture}
    \draw[red, thick] (0,2) -- (2,0);
    \draw[red, thick] (2,2) -- (0,0);
    \filldraw[black] (0,2) circle (2pt) node[anchor=east] {$u$};
    \filldraw[black] (2,2) circle (2pt) node[anchor=west] {$v$};
    \filldraw[black] (0,0) circle (2pt) node[anchor=east] {$x$};
    \filldraw[black] (2,0) circle (2pt) node[anchor=west] {$y$};
\end{tikzpicture} 
 \\ 
 \hline
\end{tabular}
\end{center}

So if there are more red edges than black edges joining between $V_B (M)$ and $V_R (M)$, we can guarantee the existence of a pair of edges, one black and one red, in $M$ such that the edges joining between them form one of the latter three (in fact two) varieties.

Using this pair of edges and appropriate order of vertices, we can make a swapping that will increase the number of red edges and decrease the number of black edges that we had from $M$ by 1 each.

Note that in any cases, we have replaced one black edge and one red edge with two red edges, so that the difference between the numbers of black edges and red edges that we had from $M$ will change by 2 in the resulting matching. If initially there are more black edges than red edges in $M$, the difference will reduce by 2. But if initially there are more red edges than black edges in $M$, the difference will increase by 2.
\end{proof}

\begin{remark}
If we are to read the proof of this lemma with the colour red and black in place of each other, the same thing will happen to the colour black when there are more black edges than red edges joining between $V_B (M)$ and $V_R (M)$.
\end{remark}

We are now ready to prove Theorem \ref{t2}.

\begin{proof}[Proof of Theorem \ref{t2}]

In our attempt to prove this statement, we will first take an arbitrary matching of $K_{4n}$, then gradually reduce the difference between the numbers of edges of each colour.

To achieve that, we will construct a finite sequence of matchings in $\mathcal{M}(G)$ which as our sequence progress $\abs{b(M)-r(M)}$, the difference between the numbers of edges of each colour, will gradually and strictly decrease until it reaches zero.

To start the proof, we first pick an arbitrary matching $M$ of $G$

We take this $M$ as $M_0$, the zeroth term of our sequence.

Next, we proceed to obtain next terms of our sequence by the following method.

For nonnegative integer $i$, if $M_i$ is a term in our sequence, it is without loss of generality to assume that $b(M_i)\geqslant r(M_i)$.

Now we have three cases to consider.

\textbf{Case 1:} $b(M_i)=r(M_i)$

In this case, we end our sequence and take $M_i$ as the matching we have been looking for.

\textbf{Case 2:} $b(M_i)-r(M_i)>2$

\textbf{Case 2.1:} $G[V_B (M_i)]$ is monochromatic.

We claim that there are more red edges than black edges joining between $V_B (M_i)$ and $V_R (M_i)$.

Since $b(M_i)>r(M_i)$, $\abs{V_B (M_i)}>\abs{V_R (M_i)}$ so there are more edges in $G[V_B (M_i)]$ than in $G[V_R (M_i)]$.

Recall that in our graph, the number of red edges and the number of black edges are equal.

Since $G[V_B (M_i)]$ is monochromatic(black) and $e(G[V_B (M_i)])>e(G[V_R (M_i)])$, there must be more red edges joining between $V_B (M_i)$ and $V_R (M_i)$ than black edges.

By applying the lemma to $M_i$, we can make a swapping that will reduce the difference by 2.

We take the resulting matching of this swapping as $M_{i+1}$ in our sequence.

\textbf{Case 2.2:} $G[V_B (M_i)]$ is not monochromatic.

Since $G[V_B (M_i)]$ is not monochromatic, there is a red edge, $ux$, in $G[V_B (M_i)]$.

Since $u,x\in V_B (M_i)$ and $ux$ is red, there must be $v,y\in V_B (M_i)$ such that $uv,xy\in M_i$.

We take $S(M_i,u,v,x,y)$ to be $M_{i+1}$ in our sequence.

This resulting matching will reduce the difference between the number of black edges and red edges by 2 or 4 compared to that of $M_i$, depending on the colour of $vy$ (see Figure \ref{f4}).

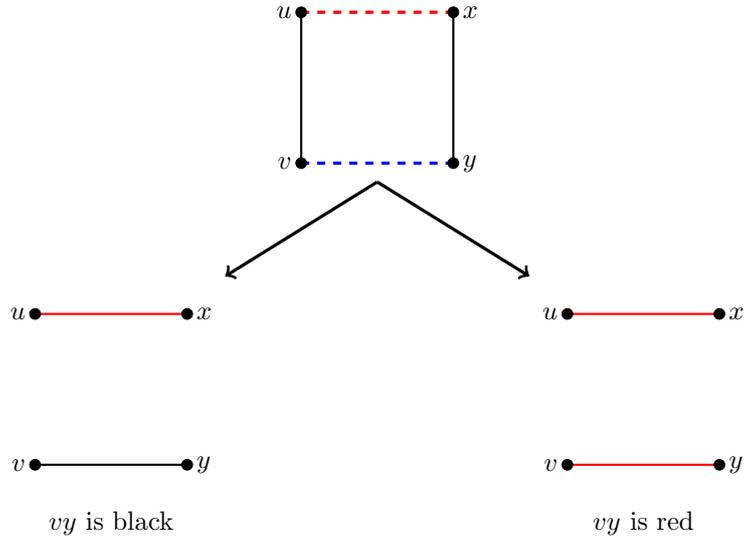
\begin{figure}[h]
    \centering
    \begin{tikzpicture}
    
    \draw[red, thick] (0,2) -- (2,2);
    \draw[black, thick] (0,0) -- (2,0);
    \filldraw[black] (0,2) circle (2pt) node[anchor=east] {$u$};
    \filldraw[black] (2,2) circle (2pt) node[anchor=west] {$x$};
    \filldraw[black] (0,0) circle (2pt) node[anchor=east] {$v$};
    \filldraw[black] (2,0) circle (2pt) node[anchor=west] {$y$};
    \filldraw[] (1,-0.5) node[anchor=north] {$vy$ is black};
    
    \draw[red, thick] (7,2) -- (9,2);
    \draw[red, thick] (7,0) -- (9,0);
    \filldraw[black] (7,2) circle (2pt) node[anchor=east] {$u$};
    \filldraw[black] (9,2) circle (2pt) node[anchor=west] {$x$};
    \filldraw[black] (7,0) circle (2pt) node[anchor=east] {$v$};
    \filldraw[black] (9,0) circle (2pt) node[anchor=west] {$y$};
    \filldraw[] (8,-0.5) node[anchor=north] {$vy$ is red};
    
    \draw[black, thick] (3.5,6) -- (3.5,4);
    \draw[black, thick] (5.5,4) -- (5.5,6);
    \draw[red,dashed,very thick] (3.5,6) -- (5.5,6);
    \draw[blue,dashed,very thick] (3.5,4) -- (5.5,4);
    \filldraw[black] (3.5,6) circle (2pt) node[anchor=east] {$u$};
    \filldraw[black] (5.5,6) circle (2pt) node[anchor=west] {$x$};
    \filldraw[black] (3.5,4) circle (2pt) node[anchor=east] {$v$};
    \filldraw[black] (5.5,4) circle (2pt) node[anchor=west] {$y$};
    
    \draw[->, very thick] (4.5,3.75) -- (2.5,2.5); 
    \draw[->, very thick] (4.5,3.75) -- (6.5,2.5); 
    
    \filldraw[] (4.5,-1.5) node[anchor=north] {If $vy$ is black, the difference reduces by 2. If $vy$ is red, the difference reduces by 4.};
    
\end{tikzpicture}

    \caption{Two possibilities of swapping.}
    
    \label{f4}
    
\end{figure}

\textbf{Case 3:} $b(M_i)-r(M_i)=2$

\textbf{Case 3.1:} $G[V_B (M_i)]$ is monochromatic.

The reasoning and execution of this case are the same as the case 2.1.

\textbf{Case 3.2:} $G[V_B (M_i)]$ is not monochromatic.

\textbf{Case 3.2.1:} There are more red edges than black edges joining between $V_B (M_i)$ and $V_R (M_i)$.

By applying the lemma to $M_i$, we can make a swapping that will reduce the difference by 2.

We take the resulting matching of this swapping as $M_{i+1}$ in our sequence.

\textbf{Case 3.2.2:} There are not more red edges than black edges joining between $V_B (M_i)$ and $V_R (M_i)$.

\textbf{Case 3.2.2.1:} There are $u,v,x,y\in V_B (M_i)$ such that $uv,xy\in E(M_i)$, $ux$ is red and $vy$ is black in $G$.

This case is the same as case 2.2 where $vy$ is black. We take $S(M_i,u,v,x,y)$ to be $M_{i+1}$.

This reduce the difference by 2, so that it becomes 0.

\textbf{Case 3.2.2.2:} For all $u,v,x,y\in V_B (M_i)$ such that $uv,xy\in E(M_i)$, if $ux$ is red then $vy$ is also red in $G$

We observe that in this case red edges always appear in pairs and each red edge only involve two edges of $M_i$, namely those that share a vertex with it. So if we count the number of red edges in $G[V_B (M_i)]$, it must be an even number.

As it will become important in the rest of the proof, let us make explicit that in this case there must be $n+1$ black edges and $n-1$ red edges in $M_i$.

So $\abs{V_B (M_i)}=2n+2$, $\abs{V_R (M_i)}=2n-2$, $e(G[V_B (M_i)])={{2n+2}\choose2}=2n^2+3n+1$, $e(G[V_R (M_i)])={{2n-2}\choose2}=2n^2-5n+3$, the number of edges joining between $V_B (M_i)$ and $V_R (M_i)$ is $(2n+2)(2n-2)=4n^2-4$, and finally the total number of edges of each colour is $\frac{1}{2}{{4n}\choose 2}=4n^2-n$.

\textbf{Case 3.2.2.2.1:} There are equal number of red edges and black edges joining between $V_B (M_i)$ and $V_R (M_i)$.

We claim that there are an odd number of black edges in $G[V_R (M_i)]$.
Note that there are an even number of red edges in $G[V_B (M_i)]$ and an even number of edges of each colour joining $V_B (M_i)$ and $V_R (M_i)$.

Since the argument is just a simple parity analysis, to avoid a verbose and confusing argument, we present the following table as our argument.

\begin{center}
 \begin{tabular}{||c c c c||} 
 \hline
 $n$ & total black edges & $e(G[V_B (M_i)])$ & black edges in $G[V_R (M_i)]$ \\ [0.5ex] 
 \hline\hline
 even & even & odd & odd \\ 
 \hline
 odd & odd & even & odd \\ 
 \hline
\end{tabular}
\end{center}

So there are an odd number of black edges in $G[V_R (M_i)]$.

Thus there are $p,q,r,s\in V_R (M_i)$ such that $pq,rs\in E(M_i)$, $pr$ is red and $qs$ is black in $G$.

Let $M'_i=S(M_i,p,q,r,s)$. Now $M'_i$ has $n+2$ black edges and $n-2$ red edges.

Observe that after the latest swapping occurs, those vertices and edges originally in $G[V_B (M_i)]$ are all contained in $G[V_B (M'_i)]$.

As a premise of this case (3.2) states that there is a red edge in $G[V_B (M_i)]$, there must be $u,v,x,y\in V_B (M_i)\subset V_B (M'_i)$ such that $uv,xy\in E(M_i)$ and $ux,vy$ are red in G.

Since $uv,xy\in E(M'_i)$, we take $M_{i+1}$ to be $S(M'_i,u,v,x,y)$.

This reduce the difference by 2, so that it becomes 0.

\textbf{Case 3.2.2.2.2:} There are more black edges than red edges joining between $V_B (M_i)$ and $V_R (M_i)$.

For the sake of clarity of how we divide our next cases, we will consider the question "What is the least number of red edges that have to be in $G[V_B (M_i)]?$".

So we will have to maximize the number of red edges outside of $G[V_B (M_i)]$.

Thus all the edges in $G[V_R (M_i)]$ and $\frac{4n^2-4}{2}-1=2n^2-3$ edges joining between $V_B (M_i)$ and $V_R (M_i)$ has to be red. Since there are $4n^2-n$ red edges in total, there are at least $(4n^2-n)-(2n^2-5n+3)-(2n^2-3)=4n$ red edges in $G[V_B (M_i)]$

\textbf{Case 3.2.2.2.2.1:} There are more than $4n$ red edges in $G[V_B (M_i)]$

Since there are more black edges than red edges joining $V_B (M_i)$ and $V_R (M_i)$. From our lemma, there must be a black edge $uv$ and a red edge $xy$ of $M_i$ such that $S(M_i,u,v,x,y)$ will increase the number of black edges from that of $M_i$ by one.

But before we make the swapping, we consider that there are $4n$ edges adjacent to $uv$ in $G[V_B (M_i)]$. So that there is a red edge not adjacent to $uv$.

Since, in this case (3.2.2.2), a red edge implies an existence of another red edge, there are $p,q,r,s\in V_B (M_i)-\{u,v\}$ such that $pq,rs\in E(M_i)$ and $pr,qs$ are red in $G$.

Let $M'_i=S(M_i,u,v,x,y)$ so that there are $n+2$ black edges and $n-2$ red edges.

We take $S(M'_i,p,q,r,s)$ to be our $M_{i+1}$.

The first swap increase the difference by 2 to be 4, then the second swap reduce the difference by 4 to 0.

\textbf{Case 3.2.2.2.2.2:} There are exactly $4n$ red edges in $G[V_B (M_i)]$

In this case there are $\frac{4n^2-4}{2}-1=2n^2-3$ red edges joining between $V_B (M_i)$ and $V_R (M_i)$ and $G[V_R (M_i)]$ is monochromatic (red).

\textbf{Case 3.2.2.2.2.2.1:} There is a pair of edges in $M_i$ of different colours such that the edges that connect between them is of the latter three varieties shown in the proof of our lemma.

In this case we are guaranteed a swapping that will reduce the difference between the numbers of black edges and red edges by 2 to 0. We take the resulting matching of that swapping to be $M_{i+1}$

\textbf{Case 3.2.2.2.2.2.2:} All pair of edges of different colours in $M_i$ have the edge of the first three varieties connect between them.

Since there are $2n^2-1$ black edges and $2n^2-3$ red edges joining between $V_B (M_i)$ and $V_R (M_i)$ and the edges are form just the first three varieties, there is the only one possibility.

That is there is a black $uv$ and a red $xy$ connecting to each other by edges of type two, while other pairs of edges in $M_i$ are connected by edges of type three in $G$.

Without loss of generality let $ux$ be red in $G$.

\textbf{Case 3.2.2.2.2.2.2.1:} There is a red edge not adjacent to $uv$ in $G[V_B (M_i)]$

As in the case 3.2.2.2.2.1, there must be $p,q,r,s\in V_B (M_i)-\{u,v\}$ such that $pq,rs\in E(M_i)$ and $pq,rs$ are red in $G$.

As in the case 3.2.2.2.2.1, we take $S(S(M_i,u,v,y,x),p,q,r,s)$ as our $M_{i+1}$.

This effectively reduce the difference from 2 to 0.

\textbf{Case 3.2.2.2.2.2.2.2:} All $4n$ red edges are adjacent to $uv$ in $G[V_B (M_i)]$

In $G[V_B (M_i)]$, $u$ and $v$ each connecting to $2n$ vertices apart from each other.

Thus they collectively involve $4n$ edges, so that all of those edges are red.

\textbf{Case 3.2.2.2.2.2.2.2.1:} There is a member of $V_B (M_i)-\{u,v\}$ that joins with $y$ by a black edge.

Let this member of $V_B (M_i)-\{u,v\}$ be called $z$.

Since $z\in V_B (M_i)$, there is a $w\in V_B (M_i)$ such that $zw\in E(M_i)$.

Now $vw$ is red and $vy,yz,zw$ are black in $G$.

We take $S(S(M_i,u,v,x,y),v,y,w,z)$ as our $M_{i+1}$ (see Figure \ref{f5}).

First swapping does not change the difference, while the second one reduce it by 2 to 0.

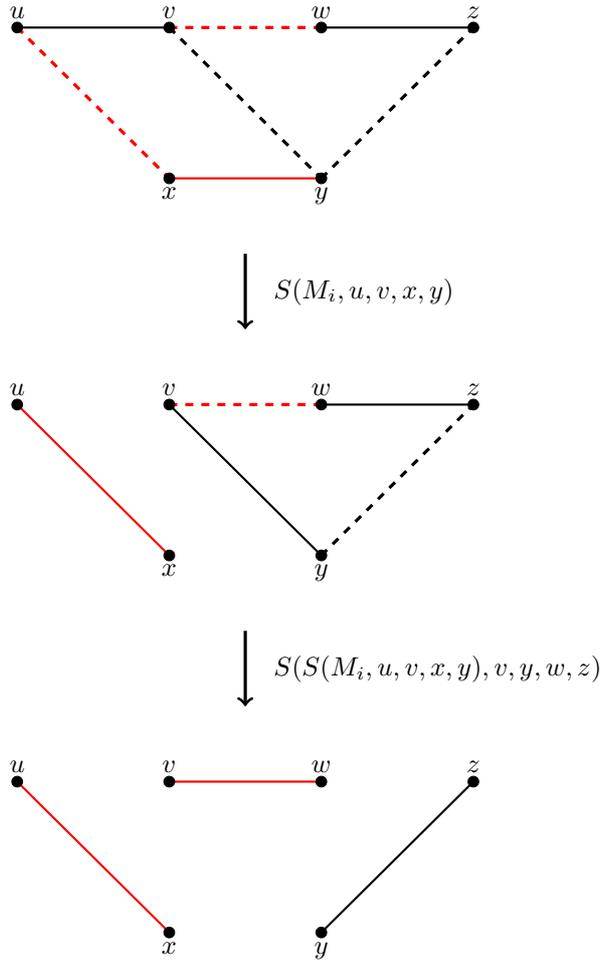
\begin{figure}[h]
    \centering
    \begin{tikzpicture}
        \draw[black,thick] (0,12) -- (2,12);
        \draw[red,dashed,very thick] (2,12) -- (4,12);
        \draw[black,thick] (4,12) -- (6,12);
        \draw[red,thick] (2,10) -- (4,10);
        \draw[red,dashed,very thick] (0,12) -- (2,10);
        \draw[black,dashed,very thick] (2,12) -- (4,10);
        \draw[black,dashed,very thick] (4,10) -- (6,12);
        \filldraw[black] (0,12) circle (2pt) node[anchor=south] {$u$};
        \filldraw[black] (2,12) circle (2pt) node[anchor=south] {$v$};
        \filldraw[black] (4,12) circle (2pt) node[anchor=south] {$w$};
        \filldraw[black] (6,12) circle (2pt) node[anchor=south] {$z$};
        \filldraw[black] (2,10) circle (2pt) node[anchor=north] {$x$};
        \filldraw[black] (4,10) circle (2pt) node[anchor=north] {$y$};
        
        \draw[->, very thick] (3,9) -- (3,8);
        \filldraw[] (3.25,8.5) node[anchor=west] {$S(M_i,u,v,x,y)$};
        
        \draw[red,dashed,very thick] (2,7) -- (4,7);
        \draw[black,thick] (4,7) -- (6,7);
        \draw[red, thick] (0,7) -- (2,5);
        \draw[black, thick] (2,7) -- (4,5);
        \draw[black,dashed,very thick] (4,5) -- (6,7);
        \filldraw[black] (0,7) circle (2pt) node[anchor=south] {$u$};
        \filldraw[black] (2,7) circle (2pt) node[anchor=south] {$v$};
        \filldraw[black] (4,7) circle (2pt) node[anchor=south] {$w$};
        \filldraw[black] (6,7) circle (2pt) node[anchor=south] {$z$};
        \filldraw[black] (2,5) circle (2pt) node[anchor=north] {$x$};
        \filldraw[black] (4,5) circle (2pt) node[anchor=north] {$y$};
        
        \draw[->, very thick] (3,4) -- (3,3);
        \filldraw[] (3.25,3.5) node[anchor=west] {$S(S(M_i,u,v,x,y),v,y,w,z)$};
        
        \draw[red, thick] (2,2) -- (4,2);
        \draw[red, thick] (0,2) -- (2,0);
        \draw[black, thick] (4,0) -- (6,2);
        \filldraw[black] (0,2) circle (2pt) node[anchor=south] {$u$};
        \filldraw[black] (2,2) circle (2pt) node[anchor=south] {$v$};
        \filldraw[black] (4,2) circle (2pt) node[anchor=south] {$w$};
        \filldraw[black] (6,2) circle (2pt) node[anchor=south] {$z$};
        \filldraw[black] (2,0) circle (2pt) node[anchor=north] {$x$};
        \filldraw[black] (4,0) circle (2pt) node[anchor=north] {$y$};
        
    \end{tikzpicture}
    \caption{Case 3.2.2.2.2.2.2.2.1}
    \label{f5}
    
\end{figure}

\textbf{Case 3.2.2.2.2.2.2.2.2:} All members of $V_B (M_i)-\{u,v\}$ joins with $y$ by red edges.

Observe that from our premises every red edges in $G[V_B (M_i)]$ have to have either $u$ or $v$ as an endpoint and that all the edge joining $u$ or $v$ with any member in $V_B (M_i)-\{u,v\}$ are red.

Now choose any $p,q\in V_B (M_i)-\{u,v\}$ the $pq$ is black while $up$ and $yq$ are red.

We take $S(S(M_i,u,v,y,x),u,y,p,q)$ as our $M_{i+1}$ (see Figure \ref{f6}).

First swapping increase the difference by 2 to be 4, then the second one reduce it by 4 to 0.

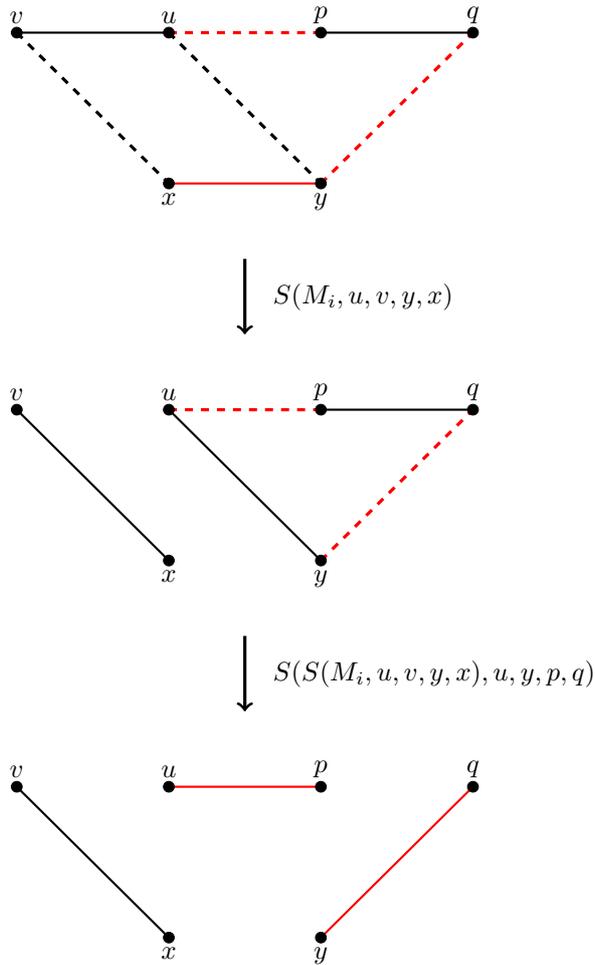
\begin{figure}[h]
    \centering
    \begin{tikzpicture}
        \draw[black,thick] (0,12) -- (2,12);
        \draw[red,dashed,very thick] (2,12) -- (4,12);
        \draw[black,thick] (4,12) -- (6,12);
        \draw[red,thick] (2,10) -- (4,10);
        \draw[black,dashed,very thick] (0,12) -- (2,10);
        \draw[black,dashed,very thick] (2,12) -- (4,10);
        \draw[red,dashed,very thick] (4,10) -- (6,12);
        \filldraw[black] (0,12) circle (2pt) node[anchor=south] {$v$};
        \filldraw[black] (2,12) circle (2pt) node[anchor=south] {$u$};
        \filldraw[black] (4,12) circle (2pt) node[anchor=south] {$p$};
        \filldraw[black] (6,12) circle (2pt) node[anchor=south] {$q$};
        \filldraw[black] (2,10) circle (2pt) node[anchor=north] {$x$};
        \filldraw[black] (4,10) circle (2pt) node[anchor=north] {$y$};
        
        \draw[->, very thick] (3,9) -- (3,8);
        \filldraw[] (3.25,8.5) node[anchor=west] {$S(M_i,u,v,y,x)$};
        
        \draw[red,dashed,very thick] (2,7) -- (4,7);
        \draw[black,thick] (4,7) -- (6,7);
        \draw[black, thick] (0,7) -- (2,5);
        \draw[black, thick] (2,7) -- (4,5);
        \draw[red,dashed,very thick] (4,5) -- (6,7);
        \filldraw[black] (0,7) circle (2pt) node[anchor=south] {$v$};
        \filldraw[black] (2,7) circle (2pt) node[anchor=south] {$u$};
        \filldraw[black] (4,7) circle (2pt) node[anchor=south] {$p$};
        \filldraw[black] (6,7) circle (2pt) node[anchor=south] {$q$};
        \filldraw[black] (2,5) circle (2pt) node[anchor=north] {$x$};
        \filldraw[black] (4,5) circle (2pt) node[anchor=north] {$y$};
        
        \draw[->, very thick] (3,4) -- (3,3);
        \filldraw[] (3.25,3.5) node[anchor=west] {$S(S(M_i,u,v,y,x),u,y,p,q)$};
        
        \draw[red, thick] (2,2) -- (4,2);
        \draw[black, thick] (0,2) -- (2,0);
        \draw[red, thick] (4,0) -- (6,2);
        \filldraw[black] (0,2) circle (2pt) node[anchor=south] {$v$};
        \filldraw[black] (2,2) circle (2pt) node[anchor=south] {$u$};
        \filldraw[black] (4,2) circle (2pt) node[anchor=south] {$p$};
        \filldraw[black] (6,2) circle (2pt) node[anchor=south] {$q$};
        \filldraw[black] (2,0) circle (2pt) node[anchor=north] {$x$};
        \filldraw[black] (4,0) circle (2pt) node[anchor=north] {$y$};
        
    \end{tikzpicture}
    \caption{Case 3.2.2.2.2.2.2.2.2}
    \label{f6}
    
\end{figure}

Now, every case, except for case 1 which is the terminal case, strictly reduces the difference between the numbers of black edges and red edges as we creating new term for our sequence.

Note that in no case that the difference were reduced by more than its value.

Thus as our sequence progress the difference strictly decreases and when it terminates, the last term of the sequence has the difference of 0.

That is by following this method, we can always guarantee a matching which has the same number of black edges and red edges.

This prove our theorem.
\end{proof}

\section{Concluding Remarks}

We have proved Theorem \ref{t2} which settles the problem posed in~\cite{caro2020zerosum}. Now we know that if we are given a $2$-edge-coloured complete graph of order $4n$ with the same number of edges of each colour, we can extract a matching that has the same number of edges of each colour.

We would like to pose two problems that are related to the result that we have proven. The first problem is a generalization of our result. In our theorem, we study only the case where the complete graph $K_{4n}$ is 2-edge-coloured and those two colours colour an equal number of edges. An obvious generalization of this problem is to consider the similar situation when there are more than two colours involved. Now we pose the following question which is a generalization of our result.

\begin{problem}
For any $k$-edge-colouring of $K_{2kn}$ such that there are an equal number of edges of each colour. Does there exist a matching such that there are an equal number of edges of each colour?
\end{problem}

The second problem comes from the fact that when we take a matching with an equal number of edges of each colour out of our original complete graph, we are left with a graph that has the same number of edges of each colour. One question that comes up is `can we take another such matching?'. If we can, can we continue until all edges are gone? If we cannot exhaust the edges with an arbitrary order of taking matchings out, is there any sequence of taking matchings out that would use every edges? This leads us to pose the following problem.

\begin{problem}
Given a $2$-edge-coloured $K_{4n}$ with an equal number of edges of each colour. Can the graph be decomposed into perfect matchings such that each matching has the same number of edges of each colour?
\end{problem}

\bibliographystyle{siam} 
\bibliography{citation}

\end{document}